\documentclass{amsart}

\usepackage{amsthm,amsmath,amssymb,amsfonts,epsfig}
\usepackage{hyperref}
\usepackage{stmaryrd}
\usepackage{epstopdf}
 \newtheorem{thm}{Theorem}[section]
 \newtheorem{cor}[thm]{Corollary}
 
 \newtheorem{prop}[thm]{Proposition}
 \theoremstyle{definition}
 \newtheorem{defn}[thm]{Definition}
 \theoremstyle{remark}
 \newtheorem{rem}[thm]{Remark}
 
 \numberwithin{equation}{section}

\newcommand{\Dom}{{\rm Dom}}

\newcommand{\Ker}{{\rm Ker}}

\renewcommand{\Re}{{\rm Re}\;}
\renewcommand{\Im}{{\rm Im}\;}

\newcommand{\dist}{{\rm dist}}

\newcommand{\ess}{\mathrm{ess}}

\newcommand{\codim}{\mathrm{codim}}

\begin{document}
\title[Spectral estimates and basis properties]
{Spectral estimates and basis properties for self-adjoint block operator matrices}
\author[Michael Strauss]{Michael Strauss}
\begin{abstract}
In the first part of this manuscript a relationship between the spectrum of self-adjoint operator matrices and the spectra of their diagonal entries is found. This leads to enclosures for spectral points and in particular, enclosures for eigenvalues. We also consider graph invariant subspaces, and their corresponding angular operators. The existence of a bounded angular operator leads to basis properties of the first component of eigenvectors of operator matrices for which the corresponding eigenvalues lie in a half line. The results are applied to an example from magnetohydrodynamics.
\end{abstract}
\maketitle
\section{Introduction}
Throughout, $A$ is a self-adjoint operator acting on a Hilbert space $\mathcal{H}_1$ and is bounded from below, $C$ is a self-adjoint operator acting on a Hilbert space $\mathcal{H}_2$ and is bounded from above, and $B$ is a densely defined closed operator from $\mathcal{H}_2$ to $\mathcal{H}_1$. We shall be concerned with the spectral properties of block operator matrices of the form
\begin{equation}\label{l0}
\mathcal{M}_0 := \left(
\begin{array}{cc}
A & B\\
B^* & C
\end{array} \right):\mathcal{H}_1\times\mathcal{H}_2\to\mathcal{H}_1\times\mathcal{H}_2
\end{equation}
with domain $\Dom(\mathcal{M}_0) = (\Dom(A)\cap\Dom(B^*))\times(\Dom(B)\cap\Dom(C))$. Operator matrices of this form appear in many applications, and the results obtained in this manuscript are applied to an example from magnetohydrodynamics. For a thorough account of this subject, the reader is referred to the monograph \cite{tret2}. We will always assume the following condition is satisfied
\begin{equation}\label{condition}
\Dom(\vert A\vert^{\frac{1}{2}})\subset\Dom(B^*).
\end{equation}
We denote by $\mathfrak{t}$ the closure of the quadratic form of $A$ which has domain $\Dom(\vert A\vert^{\frac{1}{2}})$. It follows from \eqref{condition} and the assumption that $A$ is bounded from below, that there exist constants $a,b\ge 0$ such that
\begin{equation}\label{domcons}
\Vert B^*x\Vert^2 \le a\mathfrak{t}[x] + b\Vert x\Vert^2\quad\textrm{for all}\quad x\in \Dom(\vert A\vert^{\frac{1}{2}}).
\end{equation}
We consider two classes of block operator matrices. If in addition to the condition \eqref{condition}, we have $\Dom(B)\subset\Dom(C)$, and $\Dom(B)$ is a core for $C$, then the operator $\mathcal{M}_0$ is called top-dominant (or upper-dominant). A top-dominant operator matrix is essentially self-adjoint, and the closure which we denote by $\mathcal{M}$ is given by
\begin{align}
\Dom(\mathcal{M})&=\bigg\{\left(
\begin{array}{c}
x\\
y
\end{array} \right):y\in\Dom(C),~x + \overline{(A-\upsilon I)^{-1}B}y\in\Dom(A)\bigg\}\label{top1}\\
\mathcal{M}\left(
\begin{array}{c}
x\\
y
\end{array} \right) &= \left(
\begin{array}{c}
A(x + \overline{(A-\upsilon I)^{-1}B}y) - \upsilon\overline{(A-\upsilon I)^{-1}B}y\\
B^*x + Cy
\end{array} \right)\label{top2}
\end{align}
where $\upsilon<\min\sigma(A)$ and is arbitrary; see \cite[Section 4.2]{math0} and references therein. If in addition to the condition \eqref{condition}, $\Dom(\vert C\vert^{\frac{1}{2}})\subset\Dom(B)$, then the operator $\mathcal{M}_0$ is called diagonally-dominant. A diagonally-dominant operator matrix is self-adjoint, and we therefore write $\mathcal{M}$ instead of $\mathcal{M}_0$; see \cite[Section 2]{math}.

Eigenvalues of operator matrices can often be characterised by variational principles; see for example \cite[Theorem 2.1]{math0} and \cite[Theorem 3.1]{math}. Enclosures for eigenvalues can also be obtained in terms of the eigenvalues of $A$, the upper bound on the spectrum of $C$, and the constants $a$ and $b$ which satisfy \eqref{domcons}; see \cite[Theorem 4.5]{math}. In Section 2 we prove a relationship between $\sigma(\mathcal{M})$ (the spectrum of $\mathcal{M}$) and $\sigma(A)$, $\sigma(C)$, and the constants $a$ and $b$. This leads to new enclosures for spectral points. Our approach is more general than the available variational principles; for example, we are not restricted to approximating only the discrete spectrum.

In Section 3 we consider spectral subspaces $\mathcal{L}_{(\alpha,\infty)}(\mathcal{M})$ for $\alpha\in\mathbb{R}$, that is, the range of the spectral projection associated to $\mathcal{M}$ and the interval $(\alpha,\infty)$. We are concerned with the existence of a so-called angular operator $K:\mathcal{H}_1\to\mathcal{H}_2$ with
\begin{equation}\label{theangularoperator}
\mathcal{L}_{(\alpha,\infty)}(\mathcal{M}) = \bigg\{\left(
\begin{array}{c}
x\\
Kx
\end{array} \right): x\in \Dom(K)\bigg\}.
\end{equation}
If the representation \eqref{theangularoperator} exists, the subspace $\mathcal{L}_{(\alpha,\infty)}(\mathcal{M})$ is called a graph invariant subspace. Graph invariant subspaces and angular operators have been widely studied and the following list is by no means exhaustive; \cite{adam}, \cite{amm}, \cite{KMM}, \cite{lan}, \cite{sha}, \cite{shk1}, \cite{tret1}. The case where $A$, $B$ and $C$ are bounded is considered in \cite{lan}, where it is shown that if $\Delta$ is an interval in $\rho(C)$ (the resolvent set of $C$) then the subspace $\mathcal{L}_{\Delta}(\mathcal{M})$ is graph invariant. In \cite{adam} the case where $B$ is bounded and the diagonal entries are separated is considered. It is shown that for $\max\sigma(C)< \alpha < \min\sigma(A)$, the subspace $\mathcal{L}_{(\alpha,\infty)}(\mathcal{M})$ admits the representation \eqref{theangularoperator} with $K\in\mathcal{B}(\mathcal{H}_1,\mathcal{H}_2)$. The top-dominant case with $C\in\mathcal{B}(\mathcal{H}_2)$ has been considered in \cite{sha}. The authors show that if there exists an $\alpha<\min\sigma(A)$ such that $C - \alpha I - \overline{B^*(A - \alpha I)^{-1}B} \ll 0$, then $\mathcal{L}_{(\alpha,\infty)}(\mathcal{M})$ admits the representation \eqref{theangularoperator} with $K\in\mathcal{B}(\mathcal{H}_1,\mathcal{H}_2)$. With a new approach, we prove the existence of graph invariant subspaces and bounded angular operators, our hypothesis is unrelated to those above, and is readily verified for a wide class of problems which includes any top-dominant or diagonally-dominant matrix for which $A$ has compact resolvent and $C\in\mathcal{B}(\mathcal{H}_2)$.

If $A$ has compact resolvent, then $\sigma(\mathcal{M})\cap(\max\sigma(C),\infty)$ consists only of isolated eigenvalues of finite multiplicity. In Section 4 we consider basis properties of the first component of eigenvectors corresponding eigenvalues which lie in the interval $(\max\sigma(C),\infty)$. From the existence of a bounded angular operator, one can deduce that the first components of the corresponding eigenvectors form a Riesz basis for $\mathcal{H}_1$; see \cite[Theorem 3.5]{adam}, and \cite[Corollary 2.6]{sha}. Under our hypothesis we also obtain a Riesz basis, and even a Bari basis. However, our angular operator may not be everywhere defined, in this case we find the co-dimension of $\Dom(K)$ in terms of the so-called Schur complement. In the final section we apply our results to a top-dominant operator matrix which arises in magnetohydrodynamics.

\subsection{Operator matrices and Schur complements}

Associated with $\mathcal{M}$ is the first Schur complement, which is a family of operators $S(\cdot)$ acting in $\mathcal{H}_1$. When $B$ is bounded, the first Schur complement is given by
\begin{equation}\label{shur}
S(\lambda) = A - \lambda I - B(C - \lambda I)^{-1}B^*,\quad\textrm{for}\quad\lambda\in\rho(C).
\end{equation}
Evidently, $S(\lambda)$ is defined for all $\lambda\in\rho(C)$, with $\Dom(S(\lambda)) = \Dom(A)$, and $S(\lambda)$ is self-adjoint for all $\lambda\in\rho(C)\cap\mathbb{R}$. The spectrum of $S(\lambda)$ is given by $\sigma(S) = \{\lambda\in\mathbb{C}:0\in\sigma(S(\lambda))\}$. It follows from the Schur factorisation that $\sigma(S)\cap\rho(C) = \sigma(\mathcal{M})\cap\rho(C)$; see \cite[Section 4.1]{math}. When $B$ is unbounded, the operator \eqref{shur} may not be densely defined. However, the Schur complement can be defined for any $\lambda\in\mathbb{C}$ with $\Re \lambda > \max\sigma(C)$ using the following form
\begin{displaymath}
s_\lambda(x,y) = \langle(A - \upsilon I)^{\frac{1}{2}}x,(A - \upsilon I)^{\frac{1}{2}}y\rangle + (\upsilon - \lambda)\langle x,y\rangle - \langle(C-\lambda I)^{-1}B^*x,B^*y\rangle,
\end{displaymath}
where $\Dom(s_\lambda) = \Dom(\vert A\vert^{\frac{1}{2}})$ and $\upsilon<\min\sigma(A)$ is arbitrary. The Schur complement is then defined using the first representation theorem; for more details see \cite[Section 4.2]{math0} and references therein. The spectra of the Schur complement and $\mathcal{M}$ coincide on $\{z\in\mathbb{C}:\Re z>\max\sigma(C)\}$, however, if $\upsilon<\min\sigma(A)$, $\Dom(A)\subseteq\Dom(B^*)$ and for every $a>0$ there exists a $b\ge 0$ such that
\begin{displaymath}
\Vert B^*x\Vert \le a\langle(A-\upsilon I)x,x\rangle + b\Vert x\Vert^2\quad\textrm{for all}\quad x\in\Dom(A),
\end{displaymath}
then the Schur complement can be defined on $\rho(C)$, and $\sigma(S)\cap\rho(C) = \sigma(\mathcal{M})\cap\rho(C)$; see \cite[Section 4.2]{math0}. We will find it useful to note that for any $\alpha\in\rho(\mathcal{M})$ for which $S(\alpha)$ is defined, the restriction of $(\mathcal{M}-\alpha I)^{-1}$ to $\mathcal{H}_1\times\Dom(F(\alpha))$ with $F(\alpha) = B(C-\alpha I)^{-1}$, is given by
\begin{equation}\label{resolvent}
\left(
\begin{array}{cc}
S(\alpha)^{-1} & -S(\alpha)^{-1}F(\alpha)\\
-(C - \alpha I)^{-1}B^*S(\alpha)^{-1} & (C - \alpha I)^{-1} + (C - \alpha I)^{-1}B^*S(\alpha)^{-1}F(\alpha)
\end{array} \right).
\end{equation}

\section{Resolvent sets and spectral enclosures}\label{enclosures}
Let $U$ be a subset of $\rho(C)$ on which $\sigma(S)$ and $\sigma(\mathcal{M})$ coincide. If $B$ is bounded, and $\lambda\in\sigma(\mathcal{M})\cap U$, then for each $\varepsilon>0$ there exists a normalised $x\in\Dom(A)$ such that
\begin{displaymath}
\Vert Ax - \lambda x - B(C - \lambda I)^{-1}B^*x\Vert \le \varepsilon
\end{displaymath}
and therefore
\begin{align*}
\Vert Ax - \lambda x\Vert &\le \Vert B(C - \lambda I)^{-1}B^*x\Vert + \varepsilon\\
&\le\frac{\Vert B\Vert^2}{\dist[\lambda,\sigma(C)]} + \varepsilon.
\end{align*}
From the spectral theorem it follows that
\begin{displaymath}
\dist[\lambda,\sigma(A)] \le \frac{\Vert B\Vert^2}{\dist[\lambda,\sigma(C)]}.
\end{displaymath}
The theorem below generalises this statement to the case where $B$ is unbounded.
We note that when $B$ is bounded \eqref{domcons} holds with the constants $a = 0$ and $b = \Vert B\Vert^2$.

\begin{thm}\label{thm0}
Let $\mathcal{M}$ be a top-dominant or diagonally-dominant operator matrix and let $a,b\ge0$ satisfy \eqref{domcons}. If $\lambda\in\sigma(\mathcal{M})\cap U$ and $\dist[\lambda,\sigma(C)]>a$, then
\begin{equation}\label{new}
\dist[\lambda,\sigma(A)]\le\frac{\vert a\lambda + b\vert}{\dist[\lambda,\sigma(C)] - a}.
\end{equation}
\end{thm}
\begin{proof}
Since $\lambda\notin\sigma(C)$ and $\lambda\in U$, it follows that $S(\lambda)$ is well defined and $0\in\sigma(S(\lambda))$, that is, $S(\lambda)$ does not have a bounded inverse. If $\lambda\in\sigma(A)$ then the assertion is obvious. Suppose that $\lambda\notin\sigma(A)$ and that \eqref{new} is false. With $\mathfrak{t}$ the quadratic form associated to $A$, and using \eqref{domcons}, we have for all $x\in \Dom (\vert A\vert^{\frac{1}{2}})$
\begin{align*}
\vert\langle(C - \lambda I)^{-1}B^*x,B^*x\rangle\vert
&\le\frac{\Vert B^*x\Vert^2}{\dist[\lambda,\sigma(C)]}\\
&\le\frac{a\mathfrak{t}[x]}{\dist[\lambda,\sigma(C)]} + \frac{b\langle x,x\rangle}{\dist[\lambda,\sigma(C)]}\\
&=\frac{a\mathfrak{t}[x] - a\lambda\langle x,x\rangle}{\dist[\lambda,\sigma(C)]} + \frac{a\lambda\langle x,x\rangle+b\langle x,x\rangle}{\dist[\lambda,\sigma(C)]}\\
&\le\frac{a\langle\vert A - \lambda I\vert^{\frac{1}{2}}x,\vert A - \lambda I\vert^{\frac{1}{2}}x\rangle}{\dist[\lambda,\sigma(C)]}\\
&+~ \frac{\vert a\lambda + b\vert\langle\vert A - \lambda I\vert^{\frac{1}{2}}x,\vert A - \lambda I\vert^{\frac{1}{2}}x\rangle}{\dist[\lambda,\sigma(A)]\dist[\lambda,\sigma(C)]}.
\end{align*}
That is,
\begin{displaymath}
\vert\langle(C - \lambda I)^{-1}B^*x,B^*x\rangle\vert
\le \delta\langle\vert A - \lambda I\vert^{\frac{1}{2}}x,\vert A - \lambda I\vert^{\frac{1}{2}}x\rangle,
\end{displaymath}
where
\begin{displaymath}
\delta = \frac{a}{\dist[\lambda,\sigma(C)]} + \frac{\vert a\lambda + b\vert}{\dist[\lambda,\sigma(A)]\dist[\lambda,\sigma(C)]}.
\end{displaymath}
Since \eqref{new} is assumed to be false, we have $\delta<1$. By \cite[Lemma VI.3.1]{katopert} there exists an operator $G\in\mathcal{B}(\mathcal{H}_1)$ with $\Vert G\Vert\le\delta$, such that for all $x,y\in \Dom (\vert A\vert^{\frac{1}{2}})$ we have
\begin{displaymath}
-\langle(C - \lambda I)^{-1}B^*x,B^*y\rangle = \langle G\vert A - \lambda I\vert^{\frac{1}{2}}x,\vert A - \lambda I\vert^{\frac{1}{2}}y\rangle.
\end{displaymath}
Consider the operator $T$ defined as follows,
\begin{align*}
T &= \vert A - \lambda I\vert^\frac{1}{2}\big((A - \lambda I)\vert A - \lambda I\vert^{-1} + G\big)\vert A - \lambda I\vert^\frac{1}{2}\\
&= \vert A - \lambda I\vert^\frac{1}{2}\big(I + G\vert A - \lambda I\vert(A - \lambda I)^{-1}\big)(A - \lambda I)\vert A - \lambda I\vert^{-1}\vert A - \lambda I\vert^\frac{1}{2},
\end{align*}
note that $\Dom (T)\subseteq \Dom (\vert A\vert^{\frac{1}{2}})$. Since
\begin{displaymath}
\Vert G\vert A - \lambda I\vert(A - \lambda I)^{-1}\Vert \le \delta\Vert\vert A - \lambda I\vert(A - \lambda I)^{-1}\Vert < 1,
\end{displaymath}
we deduce that $(I + G\vert A - \lambda I\vert(A - \lambda I)^{-1})^{-1}$ exists and is bounded. It follows that $T^{-1}$ exists and is bounded, that is,
\begin{displaymath}
T^{-1} = \vert A - \lambda I\vert^{-\frac{1}{2}}\vert A - \lambda I\vert(A - \lambda I)^{-1}\big(I + G\vert A - \lambda I\vert(A - \lambda I)^{-1}\big)^{-1}\vert A - \lambda I\vert^{-\frac{1}{2}}.
\end{displaymath}
The Schur complement $S(\lambda)$ is the unique operator which satisfies $\Dom (S(\lambda))\subseteq \Dom (\vert A\vert^{\frac{1}{2}})$ and
$\langle S(\lambda)x,y\rangle = s_\lambda(x,y)$ for all $x\in \Dom (S(\lambda))$ and $y\in \Dom (\vert A\vert^{\frac{1}{2}})$. Also, $\Dom (T)\subseteq \Dom (\vert A\vert^{\frac{1}{2}})$, and for all $x\in \Dom (T)$ and $y\in \Dom (\vert A\vert^{\frac{1}{2}})$, we have
\begin{align*}
\langle Tx,y\rangle &= \langle\vert A - \lambda I\vert^{\frac{1}{2}}((A - \lambda I)\vert A - \lambda I\vert^{-1} + G)\vert A - \lambda I\vert^{\frac{1}{2}}x,y\rangle\\
&= \langle((A - \lambda I)\vert A - \lambda I\vert^{-1} + G)\vert A - \lambda I\vert^{\frac{1}{2}}x,\vert A - \lambda I\vert^{\frac{1}{2}}y\rangle\\
&= \langle(A - \lambda I)\vert A - \lambda I\vert^{-1}\vert A - \lambda I\vert^{\frac{1}{2}}x,\vert A - \lambda I\vert^{\frac{1}{2}}y\rangle\\
&+~\langle G\vert A - \lambda I\vert^{\frac{1}{2}}x,\vert A - \lambda I\vert^{\frac{1}{2}}y\rangle\\
&= \langle(A - \upsilon I)\vert A - \lambda I\vert^{-1}\vert A - \lambda I\vert^{\frac{1}{2}}x,\vert A - \lambda I\vert^{\frac{1}{2}}y\rangle\\
&+~\langle G\vert A - \lambda I\vert^{\frac{1}{2}}x,\vert A - \lambda I\vert^{\frac{1}{2}}y\rangle\\
&+~(\upsilon - \lambda)\langle\vert A - \lambda I\vert^{-1}\vert A - \lambda I\vert^{\frac{1}{2}}x,\vert A - \lambda I\vert^{\frac{1}{2}}y\rangle\\
&=\langle(A - \upsilon I)^{\frac{1}{2}}x,(A - \upsilon I)^{\frac{1}{2}}y\rangle + (\upsilon - \lambda)\langle x,y\rangle - \langle(C-\lambda I)^{-1}B^*x,B^*y\rangle\\
&= s_\lambda(x,y),
\end{align*}
from which we deduce that $T = S(\lambda)$, and therefore $S(\lambda)^{-1}$ exists and is bounded. However, $\lambda\in\sigma(\mathcal{M})\cap U$ and therefore  $\lambda\in\sigma(S)$, the result follows from the contradiction.
\end{proof}

\begin{cor}\label{c1}
Let $c= \max\sigma(C)$ and $\lambda\in\sigma(\mathcal{M})\cap(c+a,\infty)$, then
\begin{equation}\label{newe}
\dist[\lambda,\sigma(A)]\le\frac{a\lambda + b}{\lambda - c - a}.
\end{equation}
If $\lambda\in[\mu,\mu + r]$ where $(\mu,\mu + 2r)\subset\rho(A)$ and $\mu\in\sigma(A)$, then $\lambda\in[\alpha^-,\alpha^+]$ where
\begin{eqnarray}\label{newer1}
\alpha^\pm = \frac{\mu + c + 2a}{2} \pm \sqrt{\bigg(\frac{\mu - c}{2}\bigg)^2 + a(a + c) + b}.
\end{eqnarray}
If $\lambda\in(\mu-r,\mu]$ where $(\mu-2r,\mu)\subset\rho(A)$, $\mu\in\sigma(A)$, and $(\mu - c)^2 > 4a\mu + 4b$ then $\lambda\notin(\beta^-,\beta^+)$ where
\begin{eqnarray}\label{newer2}
\beta^\pm = \frac{\mu + c}{2} \pm \sqrt{\bigg(\frac{\mu - c}{2}\bigg)^2 - (a\mu + b)}.
\end{eqnarray}
\end{cor}
\begin{proof}
Note that $a\lambda + b\ge 0$ whenever $\lambda\ge\min\sigma(A)$, this follows from \eqref{domcons}. If $c<\min\sigma(A)$ then $(c,\min\sigma(A))\subset\rho(\mathcal{M})$ since $s_\lambda$ is a strictly positive form for $\lambda
\in(c,\min\sigma(A))$. Therefore \eqref{newe} follows from \eqref{new}. We now prove the second assertion. Using \eqref{newe} we obtain
\begin{displaymath}
\lambda - \mu = \dist[\lambda,\sigma(A)]\le\frac{a\lambda + b}{\lambda - c - a}
\end{displaymath}
and therefore
\begin{equation}\label{ne}
\lambda^2 - (\mu + c + 2a)\lambda + (c + a)\mu - b \le 0.
\end{equation}
From \eqref{ne} we deduce that $\alpha^\pm\in\mathbb{R}$ and $\lambda\in[\alpha^-,\alpha^+]$. For the final assertion, we note that $\dist[\lambda,\sigma(A)] = \mu - \lambda$, and obtain
\begin{displaymath}
\lambda^2 - (\mu + c)\lambda + (c + a)\mu + b \ge 0,
\end{displaymath}
from which $\lambda\notin(\beta^-,\beta^+)$ follows.
\end{proof}

Although $c=\max\sigma(C)$ may not be below $\min\sigma_{\ess}(\mathcal{M})$, the eigenvalues of $\mathcal{M}$ which are greater than $c$, but less than $\lambda_e=\min(\sigma_{\ess}(\mathcal{M})\cap(c,\infty))$, can often be approximated using a variational principle, which we now describe. For $\gamma\in(c,\infty)$, let $\mathcal{L}_{(-\infty,0)}(S(\gamma))$ be the spectral subspace of $S(\gamma)$ corresponding to the interval $(-\infty,0)$. We set $\kappa(\gamma) = \dim(\mathcal{L}_{(-\infty,0)}(S(\gamma)))$ and assume there exists a $\gamma\in(c,\infty)$ such that $\kappa(\gamma)<\infty$, then there exists an $\alpha$ such that $(c,\alpha)\subset\rho(\mathcal{M})$; see \cite[Theorem 3.1]{math}. Set $\kappa = \kappa(\alpha)$ and let $\lambda_1\le\lambda_2\le\cdots\le\lambda_N$, $N\in\mathbb{N}_0\cup\{\infty\}$, be eigenvalues of $\mathcal{M}$ in the interval $(c,\lambda_e)$. Let $\mu_1\le\mu_2\le\cdots\le\mu_M$, $M\in\mathbb{N}_0\cup\{\infty\}$, be the eigenvalues of $A$ below $\sigma_{\ess}(A)$, and set $\mu_k = \min\sigma_{\ess}(A)$ for $k>M$. For $n=1,\dots,N$, we have
\begin{equation}\label{maththias}
\mu_{\kappa+n}\le\lambda_n\le\frac{\mu_{\kappa+n} + c}{2} + \sqrt{\bigg(\frac{\mu_{\kappa+n} - c}{2}\bigg)^2 + a\mu_{\kappa+n} + b};
\end{equation}
see \cite[Corollary 4.1 and Theorem 4.2]{math}. The estimate \eqref{maththias} is particularly useful when $A$ has compact resolvent. In this case, $\sigma_{\ess}(\mathcal{M})\cap(c,\infty) = \emptyset$, $\sigma(\mathcal{M})\cap(c,\infty)$ consists of a sequence of eigenvalues $\lambda_n\to\infty$, and $\kappa(\gamma)<\infty$ for all $\gamma\in(c,\infty)$; see \cite[Theorem 4.5]{math0}. Corollary \ref{c1} is unlikely to offer improvements to the estimate \eqref{maththias}, however, we can obtain information about the resolvent set in the region $(\lambda_e,\infty)$ as the corollary below shows. Also, in the theorem below, we are able to obtain spectral enclosures for spectral points in the interval $(c,\infty)$, moreover, we are not restricted to approximating the discrete spectrum and we do not require the finiteness of $\kappa(\gamma)$.

\begin{cor}\label{c1a}
Let $a + c < \mu_1,\mu_2\in\sigma(A)$ and $(\mu_1,\mu_2)\subset\rho(A)$. If $\beta_2^-<(\mu_1 + \mu_2)/2$ and $\alpha_1^+<\beta_2^+$ where
\begin{align*}
\alpha_1^\pm &= \frac{\mu_1 + c + 2a}{2} \pm \sqrt{\bigg(\frac{\mu_1 - c}{2}\bigg)^2 + a(a + c) + b},\\
\beta_2^\pm &= \frac{\mu_2 + c}{2} \pm \sqrt{\bigg(\frac{\mu_2 - c}{2}\bigg)^2 - (a\mu_2 + b)},
\end{align*}
then $\mu_1\le\alpha_1^+<\beta_2^+\le\mu_2$ and $(\alpha_1^+,\beta_2^+)\subset\rho(\mathcal{M})$.
\end{cor}
\begin{proof}
If $\mu_1 >\alpha_1^+$ we have
\begin{displaymath}
\frac{\mu_1 - c}{2} - a > +\sqrt{\bigg(\frac{\mu_1 - c}{2}\bigg)^2 + a(a + c) + b},
\end{displaymath}
then squaring both sides implies that $a\mu + b < 0$ which contradicts \eqref{domcons}. That $\beta_2^+\le\mu_2$ follows immediately from the definition of $\beta_2^+$. Let $\lambda\in\sigma(\mathcal{M})$. If $\lambda\in(\mu_1,(\mu_1+\mu_2)/2]$, then $\lambda < \alpha_1^+$ by Corollary \ref{c1}. If $\lambda\in((\mu_1+\mu_2)/2,\mu_2)$ then $\dist[\lambda,\sigma(A)]=\mu_2-\lambda$, and by Corollary \ref{c1} $\lambda\notin(\beta_2^-,\beta_2^+)$. Suppose that $\lambda \le \beta_2^-$, then $\vert\mu_1-\lambda\vert<\mu_2-\lambda$ and therefore $\dist[\lambda,\sigma(A)]<\mu_2-\lambda$; from the contradiction we deduce that $\lambda\in(\beta_2^+,\mu_2)$.
\end{proof}

\begin{thm}\label{thm3}
Let $\mu_1,\mu_2,\mu_3,\mu_4\in\sigma(A)$, $\mu_1<\mu_2\le\mu_3<\mu_4$, and the pairs $\{\mu_1,\mu_2\}$, $\{\mu_3,\mu_4\}$ satisfy the hypothesis of Corollary \ref{c1a}. With
\begin{align*}
\beta_2^+ &= \frac{\mu_2 + c}{2} + \sqrt{\bigg(\frac{\mu_2 - c}{2}\bigg)^2 - (a\mu_2 + b)},\\
\alpha_3^+ &= \frac{\mu_3 + c + 2a}{2} + \sqrt{\bigg(\frac{\mu_3 - c}{2}\bigg)^2 + a(a + c) + b},
\end{align*}
we have $[\beta_2^+,\alpha_3^+]\cap\sigma(\mathcal{M})\ne\emptyset$, moreover, $\dim(\mathcal{L}_{[\beta_2^+,\alpha_3^+]}(\mathcal{M})) = \dim(\mathcal{L}_{[\beta_2^+,\alpha_3^+]}(A))$.
\end{thm}
\begin{proof}
We consider the following family of operators
\begin{equation}\label{family}
\mathcal{M}_0(t) := \left(
\begin{array}{cc}
A & tB\\
tB^* & C
\end{array} \right)\quad\textrm{for}\quad t\in[0,1].
\end{equation}
Denote the corresponding self-adjoint closures by $\mathcal{M}(t)$. Note that the matrix $\mathcal{M}_0(0)$ is diagonally-dominant and therefore $\mathcal{M}(0) = \mathcal{M}_0(0)$, moreover, we have $\dim(\mathcal{L}_{[\beta_2^+,\alpha_3^+]}(\mathcal{M}(0))) = \dim(\mathcal{L}_{[\beta_2^+,\alpha_3^+]}(A))$. We now show that $\mathcal{M}(t)$ is a holomorphic family. Let $l=(\alpha_1^+ + \beta_2^+)/2$ and $r=(\alpha_3^+ + \beta_4^+)/2$, then $l,r\in\rho(A)$. For any $t\in[0,1]$ we have
\begin{displaymath}
\Vert tB^*x\Vert^2\le\Vert B^*x\Vert^2 \le a\mathfrak{t}[x] + b\Vert x\Vert^2\quad\textrm{for all}\quad x\in \Dom(\vert A\vert^{\frac{1}{2}}),
\end{displaymath}
it therefore follows from Corollary \ref{c1a} that $l,r\in\rho(\mathcal{M}(t))$ for every $t\in[0,1]$. We denote by $S(\lambda,t)$ the Schur complement corresponding to the matrix $\mathcal{M}(t)$, that is, the operator corresponding to the form
\begin{displaymath}
s_{\lambda,t}(x,y) = \langle(A - \upsilon I)^{\frac{1}{2}}x,(A - \upsilon I)^{\frac{1}{2}}y\rangle + (\upsilon - \lambda)\langle x,y\rangle - t^2\langle(C-\lambda I)^{-1}B^*x,B^*y\rangle,
\end{displaymath}
and the first representation theorem. $S(\lambda,1)$ is a holomorphic family of type (B) with respect $\lambda\in\{z\in\mathbb{C}:\Re z > c\}$; see \cite[Theorem VII.4.2]{katopert}. Similarly, $S(\lambda,t)$ is a holomorphic family of type (B) with respect to both variables $\lambda\in\{z\in\mathbb{C}:\Re z > c\}$ and $t\in\mathbb{R}$. Let $(x,y)^T,(u,v)^T\in\mathcal{H}_1\times\Dom(B(C-\alpha I)^{-1})$, then using \eqref{resolvent} we have
\begin{align*}
\bigg\langle(\mathcal{M}(t)-lI)^{-1}\left(
\begin{array}{c}
x\\
y
\end{array} \right)&,\left(
\begin{array}{c}
u\\
v
\end{array} \right)\bigg\rangle\\
&= \langle S(l,t)^{-1}x,u\rangle
-t\langle S(l,t)^{-1}B(C - lI)^{-1}y,u\rangle\\
&-t\langle (C - lI)^{-1}B^*S(l,t)^{-1}x,v\rangle + \langle (C - lI)^{-1}y,v\rangle\\
&+t^2\langle (C - lI)^{-1}B^*S(l,t)^{-1}B(C - lI)^{-1}y,v\rangle.
\end{align*}
It follows from the definitions of top-dominant and diagonally-dominant operators, that $\mathcal{H}_1\times\Dom(B(C-\alpha I)^{-1})$ is dense in $\mathcal{H}_1\times\mathcal{H}_2$. That $\mathcal{M}(t)$ is holomorphic on $[0,1]$ follows from \cite[Theorem III.3.12 and Theorem VII.1.3]{katopert} and the fact that $S(l,t)$ is holomorphic on $t\in\mathbb{R}$. Let $E_t([\beta_2^+,\alpha_3^+])$ be the spectral measure corresponding to the operator $\mathcal{M}(t)$ and the interval $[\beta_2^+,\alpha_3^+]$. For any $t_1,t_2\in[0,1]$ we have
\begin{displaymath}
E_{t_1} - E_{t_2} = -\frac{1}{2\pi i}\int_\Gamma (\mathcal{M}(t_1) -\zeta)^{-1} - (\mathcal{M}(t_2) -\zeta)^{-1}~d\zeta
\end{displaymath}
where $\Gamma$ is a circle which passes through $l$ and $r$, from which the result follows.
\end{proof}

To obtain accurate approximation of the spectrum, particularly of the discrete spectrum, a numerical method is likely to be necessary. In the region $(\lambda_e,\infty)$ the traditional Galerkin (finite section) method is unreliable; see for example \cite[Theorem 2.1]{lesh}. There are several numerical methods available for locating spectral points in this region, for example, the second order relative spectrum and the method of Davies and Plum; see \cite{lesh} and \cite{dapl} respectively. Both of these methods can benefit significantly from a priori information about the resolvent set and the dimension of spectral subspaces. The second order relative spectrum requires solving quadratic eigenvalue problems of the form $P(\mathcal{M} - z)^2|_{\mathcal{L}}u=0$ where $P$ is the orthogonal projection onto a finite-dimensional subspace $\mathcal{L}$; see \cite{lesh}. To demonstrate, we suppose the hypothesis of Theorem \ref{thm3} holds and that in addition $\dim(\mathcal{L}_{[\beta_2^+,\alpha_3^+]}(A)) = 1$, and therefore $\dim(\mathcal{L}_{[\beta_2^+,\alpha_3^+]}(\mathcal{M})) = 1$. If we have a solution $z$ to the quadratic eigenvalue $P(\mathcal{M} - z)^2|_{\mathcal{L}}u=0$ for some subspace $\mathcal{L}$, and such that $z$ lies inside the open disc with centre $(\alpha_1^+ + \beta_4^+)/2$ and radius $(\beta_4^- - \alpha_1^+)/2$, then by \cite[Theorem 2.2 and Remark 2.3]{str} we obtain
\begin{displaymath}
\sigma(\mathcal{M})\cap\bigg[\Re z - \frac{\vert\Im z\vert^2}{\beta^+_4 -\Re z},\Re z + \frac{\vert\Im z\vert^2}{\Re z - \alpha_1^+}\bigg] = \sigma(\mathcal{M})\cap[\beta_2^+,\alpha_3^+].
\end{displaymath}

\section{Graph invariant subspaces}

With $\mathcal{E}$ the spectral measure corresponding to $\mathcal{M}$ and $\alpha\in\mathbb{R}$, let
\begin{displaymath}
\mathcal{L}_{(\alpha,\infty)}(\mathcal{M}) := \mathcal{E}((\alpha,\infty))\big(\mathcal{H}_1\times\mathcal{H}_2\big).
\end{displaymath}
This section is concerned with the existence of angular operators corresponding to subspaces $\mathcal{L}_{(\alpha,\infty)}(\mathcal{M})$, that is, operators $K_\alpha:\mathcal{H}_1\to\mathcal{H}_2$ such that
\begin{equation}\label{angularoperator}
\mathcal{L}_{(\alpha,\infty)}(\mathcal{M}) = \bigg\{\left(
\begin{array}{c}
x\\
K_\alpha x
\end{array} \right): x\in \Dom (K_\alpha)\bigg\}.
\end{equation}

\begin{thm}\label{lem}
Let $\mathcal{M}$ be a top-dominant or diagonally-dominant operator matrix and let $C\in\mathcal{B}(\mathcal{H}_2)$. If $\alpha\in\mathbb{R}$ satisfies $\max\sigma(C) = c<\alpha\in(\rho(\mathcal{M})\cap\rho(A))$ and
\begin{equation}\label{thecondition}
\delta = \frac{a}{\alpha - c} + \frac{\vert a\alpha + b\vert}{\dist[\alpha,\sigma(A)](\alpha - c)} < \frac{1}{2},
\end{equation}
then there exists a closed bounded operator $K_\alpha$ which satisfies \eqref{angularoperator}.
\end{thm}
\begin{proof}
Similarly to the proof of Theorem \ref{thm0} we have an operator $G\in\mathcal{B}(\mathcal{H}_1)$ such that $S(\alpha)^{-1}$ admits the following representation,
\begin{displaymath}
S(\alpha)^{-1} = \vert A - \alpha I\vert^{-\frac{1}{2}}\big((A - \alpha I)\vert A - \alpha I\vert^{-1} + G\big)^{-1}\vert A - \alpha I\vert^{-\frac{1}{2}}\quad\textrm{where}\quad\Vert G\Vert\le\delta.
\end{displaymath}
Note that $(A - \alpha I)\vert A - \alpha I\vert^{-1} + G\in\mathcal{B}(\mathcal{H}_1)$ and for all $x\in\mathcal{H}_1$ we have
\begin{displaymath}
\Vert(A - \alpha I)\vert A - \alpha I\vert^{-1}x + Gx\Vert\ge\Vert(A - \alpha I)\vert A - \alpha I\vert^{-1}x\Vert - \Vert Gx\Vert\ge(1-\delta)\Vert x\Vert,
\end{displaymath}
from which we obtain
\begin{align*}
\langle S(\alpha)^{-1}x,x\rangle &=\langle \vert A - \alpha I\vert^{-\frac{1}{2}}\big((A - \alpha I)\vert A - \alpha I\vert^{-1} + G\big)^{-1}\vert A - \alpha I\vert^{-\frac{1}{2}}x,x\rangle\\
&=\langle \big((A - \alpha I)\vert A - \alpha I\vert^{-1} + G\big)^{-1}\vert A - \alpha I\vert^{-\frac{1}{2}}x,\vert A - \alpha I\vert^{-\frac{1}{2}}x\rangle\\
&\le\Vert\big((A - \alpha I)\vert A - \alpha I\vert^{-1} + G\big)^{-1}\vert A - \alpha I\vert^{-\frac{1}{2}}x\Vert\Vert\vert A - \alpha I\vert^{-\frac{1}{2}}x\Vert\\
&\le\Vert\big((A - \alpha I)\vert A - \alpha I\vert^{-1} + G\big)^{-1}\Vert\Vert\vert A - \alpha I\vert^{-\frac{1}{2}}x\Vert^2\\
&\le(1-\delta)^{-1}\Vert\vert A - \alpha I\vert^{-\frac{1}{2}}x\Vert^2.
\end{align*}
Let $x\in\mathcal{H}_1$ and $y\in \Dom(B(C - \alpha I)^{-1})$. Using \eqref{resolvent} we have
\begin{align*}
\bigg\langle(\mathcal{M}-\alpha I)^{-1}\left(
\begin{array}{c}
x\\
y
\end{array} \right)&,\left(
\begin{array}{c}
x\\
y
\end{array} \right)\bigg\rangle\\
&= \langle S(\alpha)^{-1}x,x\rangle - \langle S(\alpha)^{-1}B(C - \alpha I)^{-1}y,x\rangle\\
&- \langle (C - \alpha I)^{-1}B^*S(\alpha)^{-1}x,y\rangle + \langle (C - \alpha I)^{-1}y,y\rangle\\
&+ \langle (C - \alpha I)^{-1}B^*S(\alpha)^{-1}B(C - \alpha I)^{-1}y,y\rangle\\
&=\langle S(\alpha)^{-1}x,x\rangle - \langle y,(C - \alpha I)^{-1}B^*S(\alpha)^{-1}x\rangle\\
&- \langle (C - \alpha I)^{-1}B^*S(\alpha)^{-1}x,y\rangle + \langle (C - \alpha I)^{-1}y,y\rangle\\
&+ \langle S(\alpha)^{-1}B(C - \alpha I)^{-1}y,B(C - \alpha I)^{-1}y\rangle,
\end{align*}
and for the last term on the right-hand side we have the upper bound
\begin{displaymath}
\langle S(\alpha)^{-1}B(C - \alpha I)^{-1}y,B(C - \alpha I)^{-1}y\rangle
\le(1-\delta)^{-1}\Vert\vert A - \alpha I\vert^{-\frac{1}{2}}B(C - \alpha I)^{-1}y\Vert^2.
\end{displaymath}
The operator $B^*\vert A - \alpha I\vert^{-\frac{1}{2}}$ is bounded by the closed graph theorem, and therefore also the adjoint operator is bounded, in particular $\vert A - \alpha I\vert^{-\frac{1}{2}}B$ is bounded. Similarly to the proof of Theorem \ref{thm0} we have for any $x\in\mathcal{H}_1$
\begin{align*}
\Vert B^*\vert A - \alpha I\vert^{-\frac{1}{2}}x\Vert^2 &\le a\mathfrak{t}[\vert A - \alpha I\vert^{-\frac{1}{2}}x] + b\langle\vert A - \alpha I\vert^{-\frac{1}{2}}x,\vert A - \alpha I\vert^{-\frac{1}{2}}x\rangle\\
&\le a\mathfrak{t}[\vert A - \alpha I\vert^{-\frac{1}{2}}x] - a\alpha\Vert\vert A - \alpha I\vert^{-\frac{1}{2}}x\Vert^2 + \frac{\vert a\alpha + b\vert\langle x,x\rangle}{\dist[\alpha,\sigma(A)]}\\
&\le a\langle x,x\rangle + \frac{\vert a\alpha + b\vert\langle x,x\rangle}{\dist[\alpha,\sigma(A)]}\\
&= \delta\dist[\alpha,\sigma(C)]\Vert x\Vert^2.
\end{align*}
Therefore $\Vert\vert A - \alpha I\vert^{-\frac{1}{2}}B\Vert^2\le\delta\dist[\alpha,\sigma(C)]$. With $E$ the spectral measure corresponding to $C$, we obtain
\begin{align*}
\bigg\langle(\mathcal{M}-\alpha I)^{-1}\left(
\begin{array}{c}
x\\
y
\end{array} \right)&,\left(
\begin{array}{c}
x\\
y
\end{array} \right)\bigg\rangle\\
&\le \langle S(\alpha)^{-1}x,x\rangle - 2\Re\langle y,(C - \alpha I)^{-1}B^*S(\alpha)^{-1}x\rangle\\
&+ \langle(C - \alpha I)^{-1}y,y\rangle + \frac{1}{1-\delta}\Vert\vert A - \alpha I\vert^{-\frac{1}{2}}B(C - \alpha I)^{-1}y\Vert^2\\
&\le \langle S(\alpha)^{-1}x,x\rangle - 2\Re\langle y,(C - \alpha I)^{-1}B^*S(\alpha)^{-1}x\rangle\\
&+\langle(C - \alpha I)^{-1}y,y\rangle + \frac{\delta\dist[\alpha,\sigma(C)]}{1-\delta}\Vert(C - \alpha I)^{-1}y\Vert^2\\
&=\langle S(\alpha)^{-1}x,x\rangle - 2\Re\langle y,(C - \alpha I)^{-1}B^*S(\alpha)^{-1}x\rangle\\
&-\langle(\alpha I - C)^{-1}y,y\rangle + \frac{\delta\dist[\alpha,\sigma(C)]}{1-\delta}\int_{\mathbb{R}}\frac{1}{(\alpha - \mu)^2}~d\langle E_\mu y,y\rangle\\
&\le\langle S(\alpha)^{-1}x,x\rangle - 2\Re\langle y,(C - \alpha I)^{-1}B^*S(\alpha)^{-1}x\rangle\\
&-\langle(\alpha I - C)^{-1}y,y\rangle + \frac{\delta}{1-\delta}\int_{\mathbb{R}}\frac{1}{(\alpha - \mu)}~d\langle E_\mu y,y\rangle\\
&=\langle S(\alpha)^{-1}x,x\rangle - 2\Re\langle y,(C - \alpha I)^{-1}B^*S(\alpha)^{-1}x\rangle\\
&+ \bigg(\frac{\delta}{1-\delta}-1\bigg)\langle(\alpha I - C)^{-1}y,y\rangle.
\end{align*}
We see that if $x = 0$ and $y\ne 0$, then the right-hand side is negative. Since $\Dom(B(C - \alpha I)^{-1})$ is dense in $\mathcal{H}_2$ we deduce that $\mathcal{L}_{(\alpha,\infty)}(\mathcal{M})$ is the graph of an operator, denote this operator by $K_\alpha$. The subspace $\mathcal{L}_{(\alpha,\infty)}(\mathcal{M})$ is closed, thus $K_\alpha$ is a closed operator.

We suppose that $K_\alpha$ is unbounded. Therefore, there is a sequence $x_n\to 0$ with $x_n\in\Dom(K_\alpha)$ and $\Vert K_\alpha x_n\Vert\to 1$. Let $x\in\mathcal{H}_1$, $y\in \Dom(B(C - \alpha I)^{-1})$, and note that by the closed graph theorem $(C - \alpha I)^{-1}B^*S(\alpha)^{-1}$ is bounded,
then from the above inequality we have
\begin{align*}
\bigg\langle(\mathcal{M}-\alpha I)^{-1}\left(
\begin{array}{c}
x\\
y
\end{array} \right)&,\left(
\begin{array}{c}
x\\
y
\end{array} \right)\bigg\rangle\\
&\le \Vert S(\alpha)^{-1}\Vert\Vert x\Vert^2
+ 2\Vert(C - \alpha I)^{-1}B^*S(\alpha)^{-1}\Vert\Vert x\Vert\Vert y\Vert\\
&+ \bigg(\frac{\delta}{1-\delta}-1\bigg)\langle(\alpha I - C)^{-1}y,y\rangle.
\end{align*}
Since $\Dom(B(C - \alpha I)^{-1})$ is dense in $\mathcal{H}_2$, for any $y\in\mathcal{H}_2$ there is a sequence $y_n\to y$ with $y_n\in\Dom(B(C - \alpha I)^{-1})$. Evidently, the above inequality holds for all $x\in\mathcal{H}_1$ and for all $y\in\mathcal{H}_2$. We obtain
\begin{displaymath}
\bigg\langle(\mathcal{M}-\alpha I)^{-1}\left(
\begin{array}{c}
x_n\\
K_\alpha x_n
\end{array} \right),\left(
\begin{array}{c}
x_n\\
K_\alpha x_n
\end{array} \right)\bigg\rangle
< 0\quad\textrm{for all sufficiently large }n\in\mathbb{N}.
\end{displaymath}
However, $(x_n,K_\alpha x_n)^T\in\mathcal{L}_{(\alpha,\infty)}(\mathcal{M})$ for all $n\in\mathbb{N}$ and therefore
\begin{displaymath}
\bigg\langle(\mathcal{M}-\alpha I)^{-1}\left(
\begin{array}{c}
x_n\\
K_\alpha x_n
\end{array} \right),\left(
\begin{array}{c}
x_n\\
K_\alpha x_n
\end{array} \right)\bigg\rangle
\ge 0\quad\textrm{for all }n\in\mathbb{N}.
\end{displaymath}
From the contradiction we deduce that $K_\alpha$ is bounded.
\end{proof}

\begin{rem}
If $\{\mu_1,\mu_2\}$ is a pair which satisfies Corollary \ref{c1a} then $(\alpha_1^++\beta_2^+)/2\in(\rho(\mathcal{M})\cap\rho(A))$ and a possible choice for $\alpha$ in \eqref{thecondition} would be $(\alpha_1^++\beta_2^+)/2$. Also, if $\{\mu_n,\mu_{n+1}\}$ is a sequence of such pairs and $\vert\mu_{j+1} - \mu_j\vert\to\infty$, then
\begin{displaymath}
\frac{a}{\alpha_j - c} + \frac{a\alpha_j + b}{\dist[\alpha_j,\sigma(A)](\alpha_j - c)}\to 0.
\end{displaymath}
\end{rem}

For any normalised $(x,y)^T\in\Dom(\mathcal{M})$, it follows from \eqref{top1} and the observation $(A-\upsilon I)^{-\frac{1}{2}}\overline{(A-\upsilon I)^{-\frac{1}{2}}B} = \overline{(A-\upsilon I)^{-1}B}$, that $\overline{(A - \upsilon I)^{-1}B}y\in\Dom(\vert A\vert^\frac{1}{2})$ and $x\in\Dom(\vert A\vert^\frac{1}{2})$. Noting also that the adjoint of $\overline{(A - \upsilon I)^{-\frac{1}{2}}B}$ is given by $B^*(A - \upsilon I)^{-\frac{1}{2}}$, and using \eqref{top2} and \eqref{domcons}, we have
\begin{align*}
\bigg\langle\mathcal{M}\left(
\begin{array}{c}
x\\
y
\end{array} \right)&,
\left(
\begin{array}{c}
x\\
y
\end{array} \right)\bigg\rangle\\
&=\langle (A - \upsilon)(x + \overline{(A - \upsilon I)^{-1}B}y),x\rangle + \upsilon\langle x,x\rangle + \langle B^*x,y\rangle + \langle Cy,y\rangle\\
&=\langle (A - \upsilon)^\frac{1}{2}x,(A - \upsilon)^\frac{1}{2}x\rangle + \upsilon\langle x,x\rangle + \langle B^*x,y\rangle + \langle Cy,y\rangle\\
&+ \langle (A - \upsilon)^\frac{1}{2}\overline{(A - \upsilon I)^{-1}B}y,(A - \upsilon)^\frac{1}{2}x\rangle\\
&= \mathfrak{t}[x] + \langle B^*x,y\rangle + \langle Cy,y\rangle
+ \langle \overline{(A - \upsilon I)^{-\frac{1}{2}}B}y,(A - \upsilon)^\frac{1}{2}x\rangle\\
&=\mathfrak{t}[x] + \langle B^*x,y\rangle + \langle y,B^*x\rangle + \langle Cy,y\rangle\\
&\ge \mathfrak{t}[x] - 2\Vert B^*x\Vert + \langle Cy,y\rangle\\
&\ge \mathfrak{t}[x] - 2\sqrt{a\mathfrak{t}[x] + b\Vert x\Vert^2} + \langle Cy,y\rangle.
\end{align*}
If $C\in\mathcal{B}(\mathcal{H}_2)$ it follows that $\mathcal{M}$ is bounded from below and therefore form closable, the corresponding closed form we denote by $\mathfrak{a}$. Similarly, we see that $\Dom(\mathfrak{a}) = \Dom(\vert A\vert^{\frac{1}{2}})\times\mathcal{H}_2$, and for any $(x,y)^T\in\Dom(\mathfrak{a})$ we have
\begin{equation}\label{quadform}
\mathfrak{a}\left[\left(
\begin{array}{c}
x\\
y
\end{array} \right)\right] =
\mathfrak{t}[x] + \langle B^*x,y\rangle + \langle y,B^*x\rangle + \langle Cy,y\rangle.
\end{equation}

\begin{cor}\label{c3}
Let $\mathcal{M}$ be a top-dominant or diagonally-dominant operator matrix and let $C\in\mathcal{B}(\mathcal{H}_2)$. Suppose that for some $c<\alpha\in\rho(\mathcal{M})$ there exists a closed bounded operator $K_\alpha$ which satisfies \eqref{angularoperator}. If $\dim(\mathcal{L}_{(-\infty,0)}(S(\gamma)))<\infty$ for some $\gamma\in(c,\infty)$, and $(c,\alpha)\cap\sigma_{\ess}(S)=\emptyset$, then there exists a $\tilde{c}\in\mathbb{R}$ with $(c,\tilde{c}]\subset\rho(\mathcal{M})$, and a closed bounded operator $K_c$ which satisfies \eqref{angularoperator}, moreover, $\codim(\Dom(K_c)) =\kappa:= \dim(\mathcal{L}_{(-\infty,0)}(S(\tilde{c}))) <\infty$.
\end{cor}
\begin{proof}
Let $\lambda_1<\dots<\lambda_n$ be the eigenvalues of $\mathcal{M}$ which lie in the interval $(c,\alpha)$. Let $x_1,\dots,x_m$ be an orthonormal basis of $\Ker(S(\lambda_n))$, and suppose that $\alpha_1x_1 + \dots + \alpha_m x_m = \tilde{x}\in\Dom(K_\alpha)\backslash\{0\}$. Set $\mathcal{L} = \mathcal{L}_{(-\infty,0]}(S(\lambda_n))$. Note that $\mathcal{L}$ is a finite-dimensional subspace of $\Dom(\vert A\vert^{\frac{1}{2}})$ and therefore $(\tilde{x},K_\alpha\tilde{x})^T\in\Dom(\mathfrak{a})$. Using \cite[Lemma 3.5]{math}, \cite[proof of Theorem 2.1]{math0} and \eqref{quadform} we have
\begin{align*}
\alpha &< \mathfrak{a}\left[\left(
\begin{array}{c}
\tilde{x}\\
K_\alpha\tilde{x}
\end{array} \right)\right]\bigg/\left\Vert\left(
\begin{array}{c}
\tilde{x}\\
K_\alpha\tilde{x}
\end{array} \right)\right\Vert^2\\
&= \frac{\mathfrak{t}[\tilde{x}] + \langle B^*\tilde{x},K_\alpha\tilde{x}\rangle + \langle K_\alpha\tilde{x},B^*\tilde{x}\rangle + \langle CK_\alpha\tilde{x},K_\alpha\tilde{x}\rangle}{\Vert \tilde{x}\Vert^2 + \Vert K_\alpha\tilde{x}\Vert^2}\\
&\le \max_{x\in\mathcal{L}\backslash\{0\}}\max_{y\in\mathcal{H}_2}\bigg\{\frac{\mathfrak{t}[x] + \langle B^*x,y\rangle + \langle y,B^*x\rangle + \langle Cy,y\rangle}{\Vert x\Vert^2 + \Vert y\Vert^2}\bigg\}\\&=\lambda_n,
\end{align*}
however, $\alpha>\lambda_n$, and from the contradiction we deduce that $\mathcal{L}_{[\lambda_{n},\infty)}(\mathcal{M})$ is also the graph of an operator. This operator is a finite-dimensional extension of $K_\alpha$, and therefore is also closed and bounded. We repeat this argument for each of the eigenvalues $\lambda_1,\dots,\lambda_n$  which lie in the interval $(c,\alpha)$ and deduce that $K_\alpha$ has a finite-dimensional extension to a closed bounded operator $K_c$ such that \eqref{angularoperator} holds. It follows that $\Dom(K_c)$ is closed, and therefore $\codim(\Dom(K_c))$ is equal to $\dim(\Dom(K_c)^\perp)$; see \cite[Lemma III.1.40]{katopert}. For proof of the existence of a $\tilde{c}$ with the stated properties see \cite[Theorem 3.1]{math}. Using \cite[Lemma 3.4]{math} we have for any $x\in\mathcal{L}_{(-\infty,0)}(S(\tilde{c}))$ and for any $y\in\mathcal{H}_2\backslash\{0\}$
\begin{equation}\label{k}
\mathfrak{a}\left[\left(
\begin{array}{c}
x\\
y
\end{array} \right)\right] =  \mathfrak{t}[x]  + \langle B^*x,y\rangle + \langle y,B^*x\rangle + \langle Cy,y\rangle \le \tilde{c}\left\Vert\left(
\begin{array}{c}
x\\
y
\end{array} \right)\right\Vert^2.
\end{equation}
If $\codim(\Dom(K_c))<\dim(\mathcal{L}_{(-\infty,0)}(S(\tilde{c})))$, then there exists an $x\in\mathcal{L}_{(-\infty,0)}(S(\tilde{c}))$ such that $x\perp\Dom(K_c)^\perp$, that is, $x\in\Dom(K_c)$. Therefore $(x,K_cx)^T\in\Dom(\mathfrak{a})$, $(x,K_cx)^T\in\mathcal{L}_{(c,\infty)}(\mathcal{M})$, and
\begin{displaymath}
\tilde{c}\left\Vert\left(
\begin{array}{c}
x\\
K_cx
\end{array} \right)\right\Vert^2 < \mathfrak{a}\left[\left(
\begin{array}{c}
x\\
K_cx
\end{array} \right)\right] =  \mathfrak{t}[x]  + \langle B^*x,K_cx\rangle + \langle K_cx,B^*x\rangle + \langle CK_cx,K_cx\rangle.
\end{displaymath}
However, if $0\ne y_n\to K_cx$ we have
\begin{displaymath}
\mathfrak{a}\left[\left(
\begin{array}{c}
x\\
y_n
\end{array} \right)\right] \to \mathfrak{a}\left[\left(
\begin{array}{c}
x\\
K_cx
\end{array} \right)\right] > \tilde{c}\left\Vert\left(
\begin{array}{c}
x\\
K_cx
\end{array} \right)\right\Vert^2,
\end{displaymath}
contradicting \eqref{k}, from which we deduce that $\codim(\Dom(K_c)) \ge \kappa$. Let $x\perp\Dom(K_c)$, then $(x,0)^T\perp\mathcal{L}_{(c,\infty)}(\mathcal{M})$, and using \eqref{resolvent}
\begin{displaymath}
0 > \bigg\langle(\mathcal{M}- \tilde{c}I)^{-1}\left(
\begin{array}{c}
x\\
0
\end{array} \right),
\left(
\begin{array}{c}
x\\
0
\end{array} \right)\bigg\rangle =\langle S(\tilde{c})^{-1}x,x\rangle,
\end{displaymath}
from which we deduce that $\codim(\Dom(K_c))\le\kappa$.
\end{proof}

\begin{thm}\label{compact}
Let $\mathcal{M}$ be a top-dominant or diagonally-dominant operator matrix with $A$ unbounded and
$C\in\mathcal{B}(\mathcal{H}_2)$. If there exists an $\alpha\in(c,\infty)\cap\rho(\mathcal{M})$ such that
$a/(\alpha - c) < 1/2$, then for sufficiently large $\beta>c$ there exists a closed bounded operator $K_\beta$ which satisfies \eqref{angularoperator}.
\end{thm}
\begin{proof}
Choose $\mu\in\sigma(A)\cap(\alpha,\infty)$, such that
\begin{equation}\label{condi}
\frac{a}{\alpha - c} + \frac{\vert a\alpha + b\vert}{(\mu - \alpha)(\alpha - c)}
< \frac{1}{2}.
\end{equation}
Let E be the spectral measure associated to $A$. We consider the operator matrix
\begin{equation}\label{family2}
\tilde{\mathcal{M}}_0 := \left(
\begin{array}{cc}
A + tE((-\infty,\mu)) & B\\
B^* & C
\end{array} \right)\quad\textrm{where}\quad t=\mu-\min\sigma(A).
\end{equation}
Let $\tilde{\mathcal{M}}$ be the self-adjoint closure of \eqref{family2}. We have $A + tE((-\infty,\mu))\ge\mu I > \alpha I > cI$. It follows that the Schur complement associated to $\tilde{\mathcal{M}}$ is strictly positive on $(c,\mu)$, therefore $(c,\alpha]\subset\rho(\tilde{\mathcal{M}})$. Also note that $\min\sigma(\mathcal{M})\le\min\sigma({\tilde{\mathcal{M}}})$, and
\begin{displaymath}
\tilde{\mathcal{M}} = \mathcal{M} + t\mathcal{P},\quad\textrm{where}\quad \mathcal{P} =
\left(
\begin{array}{cc}
E((-\infty,\mu)) & 0\\
0 & 0
\end{array} \right)\quad\textrm{and}\quad\Vert\mathcal{P}\Vert = 1.
\end{displaymath}
Denote by $\tilde{\mathfrak{t}}$ the form associated to the operator $A + tE((-\infty,\mu))$, then it is clear that $\Dom(\tilde{\mathfrak{t}}) = \Dom(\mathfrak{t}) = \Dom(\vert A\vert^{\frac{1}{2}})$. Moreover, for all $x\in\Dom(\vert A\vert^{\frac{1}{2}})$ we have
\begin{displaymath}
\Vert B^*x\Vert^2 \le a\mathfrak{t}[x] + b\Vert x\Vert^2 \le a\tilde{\mathfrak{t}}[x] + b\Vert x\Vert^2,
\end{displaymath}
that is, the constants $a,b\in\mathbb{R}$ which satisfy \eqref{domcons}, also satisfy \eqref{domcons} when we replace $A$ with $A + tE((-\infty,\mu))$. We have $A + tE((-\infty,\mu))\ge\mu I$, therefore $\dist[\alpha,\sigma(A + tE((-\infty,\mu)))] \ge \mu - \alpha$, and we obtain
\begin{displaymath}
\frac{a}{\alpha - c} + \frac{\vert a\alpha + b\vert}{\dist[\alpha,\sigma(A + tE((-\infty,\mu)))](\alpha - c)}\le\frac{a}{\alpha - c} + \frac{\vert a\alpha + b\vert}{(\mu-\alpha)(\alpha - c)}
< \frac{1}{2}.
\end{displaymath}
Using Theorem \ref{lem} we deduce that there exists a closed bounded operator $\tilde{K}_\alpha$ which satisfies \eqref{angularoperator} for the subspace $\mathcal{L}_{(\alpha,\infty)}(\tilde{\mathcal{M}})$. Since $(c,\alpha]\subset\rho(\tilde{\mathcal{M}})$, we have $\mathcal{L}_{(c,\infty)}(\tilde{\mathcal{M}}) = \mathcal{L}_{(\alpha,\infty)}(\tilde{\mathcal{M}})$, that is, we have a closed bounded operator $\tilde{K}_c$ which satisfies \eqref{angularoperator} for the subspace $\mathcal{L}_{(c,\infty)}(\tilde{\mathcal{M}})$.

Let $\mathcal{E}$ and $\tilde{\mathcal{E}}$ be the spectral measures associated to $\mathcal{M}$ and
$\tilde{\mathcal{M}}$ respectively. Let $\beta>\alpha$, and $\Gamma$ be a circle which passes through $\alpha$ and any $\tilde{\alpha}<\min\sigma(\mathcal{M})$. For any normalised $\phi\in\mathcal{L}_{(\beta,\infty)}(\mathcal{M})$ we have
\begin{align*}
\Vert\tilde{\mathcal{E}}((-\infty,\alpha))\phi\Vert^2 &= \langle \tilde{\mathcal{E}}((-\infty,\alpha))\phi,\tilde{\mathcal{E}}((-\infty,\alpha))\phi\rangle\\
&= \langle \tilde{\mathcal{E}}((-\infty,\alpha))\phi,\phi\rangle\\
&= \langle [\tilde{\mathcal{E}}((-\infty,\alpha)) - \mathcal{E}((-\infty,\alpha))]\phi,\phi\rangle\\
&= -\frac{1}{2i\pi}\int_\Gamma \langle[(\tilde{\mathcal{M}}-\zeta I)^{-1} -
(\mathcal{M}-\zeta I)^{-1}]\phi,\phi\rangle~d\zeta.
\end{align*}
Since $(\tilde{\mathcal{M}}-\zeta I)^{-1} -
(\mathcal{M}-\zeta)^{-1} = (\mathcal{M}-\zeta I)^{-1}t\mathcal{P}(\tilde{\mathcal{M}}-\zeta I)^{-1}$ we have
\begin{align*}
\Vert\tilde{\mathcal{E}}(-\infty,\alpha))\phi\Vert^2
&=-\frac{1}{2i\pi}\int_\Gamma
\langle[(\mathcal{M}-\zeta I)^{-1}t\mathcal{P}(\tilde{\mathcal{M}}-\zeta I)^{-1}\phi,\phi\rangle~d\zeta\\
&\le\frac{1}{2\pi}\int_\Gamma
\vert\langle[(\mathcal{M}-\zeta I)^{-1}t\mathcal{P}(\tilde{\mathcal{M}}-\zeta I)^{-1}\phi,\phi\rangle\vert~d\zeta\\
&=\frac{1}{2\pi}\int_\Gamma
\vert\langle t\mathcal{P}(\tilde{\mathcal{M}}-\zeta I)^{-1}\phi,(\mathcal{M}-\overline{\zeta}I)^{-1}\phi\rangle\vert~d\zeta\\
&\le\frac{t(\alpha-\tilde{\alpha})}{2}\max_{z\in\Gamma}\Vert(\tilde{\mathcal{M}}-zI)^{-1}\Vert\max_{w\in\Gamma}\Vert
(\mathcal{M}-wI)^{-1}\phi\Vert.
\end{align*}
Let $M = \max_{z\in\Gamma}\Vert(\tilde{\mathcal{M}}-zI)^{-1}\Vert$, and note that
\begin{align*}
\max_{w\in\Gamma}\Vert(\mathcal{M}-wI)^{-1}\phi\Vert &= \max_{w\in\Gamma}\Big(\int_\mathbb{R}\frac{1}{\vert\lambda - w\vert^2}~d\langle \mathcal{E}_\lambda\phi,\phi\rangle\Big)^{\frac{1}{2}}\\
&= \max_{w\in\Gamma}\Big(\int_\beta^\infty\frac{1}{\vert\lambda - w\vert^2}~d\langle \mathcal{E}_\lambda\phi,\phi\rangle\Big)^{\frac{1}{2}}\\
&\le\frac{1}{\beta - \alpha}.
\end{align*}
We obtain
\begin{displaymath}
\Vert\tilde{\mathcal{E}}(-\infty,\alpha))\phi\Vert\le\sqrt{\frac{t(\alpha-\tilde{\alpha}) M}{2(\beta-\alpha)}}.
\end{displaymath}
Let $0<\varepsilon < 1$. For sufficiently large $\beta\in\mathbb{R}$ and any $\phi\in\mathcal{L}_{(\beta,\infty)}(\mathcal{M})$ with $\Vert\phi\Vert=1$, we have $\Vert\tilde{\mathcal{E}}(-\infty,\alpha))\phi\Vert\le\varepsilon$. Suppose that $\phi$ is of the form
\begin{displaymath}
\phi = \left(
\begin{array}{c}
0\\
y
\end{array} \right)\quad\textrm{and therefore}\quad \Vert y\Vert=1.
\end{displaymath}
We have
\begin{displaymath}
\tilde{\mathcal{E}}((\alpha,\infty))\phi = \left(
\begin{array}{c}
x\\
\tilde{K}_cx
\end{array} \right)\quad\textrm{for some}\quad x\in\mathcal{H}_1,
\end{displaymath}
and
\begin{align*}
\bigg\Vert\left(
\begin{array}{c}
0\\
y
\end{array} \right) - \left(
\begin{array}{c}
x\\
\tilde{K}_cx
\end{array} \right)\bigg\Vert
&= \Vert[I - \tilde{\mathcal{E}}((\alpha,\infty))]\phi\Vert\\
&= \Vert\tilde{\mathcal{E}}((-\infty,\alpha))\phi\Vert\\
&\le\varepsilon,
\end{align*}
from which we obtain $\Vert x\Vert\le\varepsilon$ and $1 - \Vert\tilde{K}_cx\Vert = \Vert y\Vert - \Vert\tilde{K}_cx\Vert \le \Vert y -\tilde{K}_cx\Vert \le \varepsilon$, therefore $\Vert\tilde{K}_c\Vert\ge(1-\varepsilon)/\varepsilon$.
However, we may choose $\varepsilon>0$ to be arbitrarily small. We deduce that for sufficiently large $\beta\in\mathbb{R}$ there exists a closed operator $K_\beta$ which satisfies \eqref{angularoperator} for the subspace $\mathcal{L}_{(\beta,\infty)}(\mathcal{M})$. A similar argument shows that $K_\beta$ is bounded.
\end{proof}

If $\dim(\mathcal{L}_{(-\infty,0)}(S(\gamma)))<\infty$ for some $\gamma\in(c,\infty)$, and we can choose $\beta$ in Theorem \ref{compact} with $(c,\beta)\cap\sigma_{\ess}(\mathcal{M})=\emptyset$, then by Corollary \ref{c3} we can extend the operator $K_\beta$ to an operator $K_c$. In particular, we have the following corollary.

\begin{cor}\label{compactresolvent}
Let $\mathcal{M}$ be a top-dominant or diagonally-dominant operator matrix where $A$ has compact resolvent and
$C\in\mathcal{B}(\mathcal{H}_2)$. There exists a $\tilde{c}\in\mathbb{R}$ with $(c,\tilde{c}]\subset\rho(\mathcal{M})$, and a closed bounded operator $K_c$ which satisfies \eqref{angularoperator}, moreover, $\codim(\Dom(K_c))= \kappa := \dim(\mathcal{L}_{(-\infty,0)}(S(\tilde{c})))<\infty$.
\end{cor}
\begin{proof}
For a diagonally-dominant or top-dominant matrix with $C\in\mathcal{B}(\mathcal{H}_2)$
and $A$ having compact resolvent, we have $\sigma_{\ess}(\mathcal{M})\cap(c,\infty)=\emptyset$; see \cite[Theorem 4.5]{math0}. The
interval $(c,\infty)$ contains a sequence of eigenvalues of finite multiplicity which
accumulate at infinity. The result follows from Theorem \ref{compact} and Corollary \ref{c3}.
\end{proof}

The top-dominant case with $C\in\mathcal{B}(\mathcal{H}_2)$ has also been considered in \cite{sha}. We now demonstrate that our hypothesis includes operator matrices which cannot satisfy the hypothesis considered in \cite{sha}. The authors of \cite{sha} show that if there exists an $\alpha<\min\sigma(A)$ such that $C - \alpha I - \overline{B^*(A - \alpha I)^{-1}B} \ll 0$, then $\mathcal{L}_{(\alpha,\infty)}(\mathcal{M})$ is a graph invariant subspace, and the corresponding angular operator belongs to $\mathcal{B}(\mathcal{H}_1,\mathcal{H}_2)$; see \cite[Theorem 2.5]{sha}.
We suppose that the hypothesis of Corollary \ref{compactresolvent} is satisfied with $\kappa\ne 0$, and that $\min\sigma(A) < \max\sigma_{\ess}(\mathcal{M})$. Suppose also that the hypothesis of \cite[Theorem 2.5]{sha} is satisfied, that is, there exists an $\alpha<\min\sigma(A)$ with $C - \alpha I - \overline{B^*(A - \alpha I)^{-1}B} \ll 0$. By \cite[Theorem 2.5]{sha} this implies the existence of an operator $K_\alpha$ satisfying \eqref{angularoperator}. Note that
\begin{displaymath}
\mathcal{L}_{(\alpha,\infty)}(\mathcal{M}) = \mathcal{L}_{(\alpha,c]}(\mathcal{M}) \oplus \mathcal{L}_{(c,\infty)}(\mathcal{M}).
\end{displaymath}
The operator $K_c$ implied by Corollary \ref{compactresolvent} has $\codim(\Dom(K_c))= \kappa <\infty$, and is a restriction of the operator $K_\alpha$. Since $\alpha < \min\sigma(A) < \max\sigma_{\ess}(\mathcal{M}) \le c$, we have $\dim(\mathcal{L}_{(\alpha,c]}(\mathcal{M})) = \infty$. Let
\begin{displaymath}
\left(
\begin{array}{c}
x_1\\
K_\alpha x_1
\end{array} \right),\dots,
\left(
\begin{array}{c}
x_{\kappa+1}\\
K_\alpha x_{\kappa+1}
\end{array} \right)\in\mathcal{L}_{(\alpha,c]}(\mathcal{M}),
\end{displaymath}
and
\begin{displaymath}
\left(
\begin{array}{c}
x_i\\
K_\alpha x_i
\end{array} \right)\perp
\left(
\begin{array}{c}
x_j\\
K_\alpha x_j
\end{array} \right)\quad\textrm{for}\quad i\ne j.
\end{displaymath}
The $\{x_1,\dots,x_{\kappa+1}\}$ are a linearly independent set, but since $\codim(\Dom(K_c))= \kappa$ we have $0\ne x = a_1x_1+\dots+a_{\kappa+1}x_{\kappa+1}\in\Dom(K_c)$ for some choice of $a_1,\dots,a_{\kappa+1}$, and therefore
\begin{displaymath}
\left(
\begin{array}{c}
x\\
K_\alpha x
\end{array} \right)\in\mathcal{L}_{(\alpha,c]}(\mathcal{M})\quad\textrm{and}\quad
\left(
\begin{array}{c}
x\\
K_\alpha x
\end{array} \right) =
\left(
\begin{array}{c}
x\\
K_c x
\end{array} \right)\in\mathcal{L}_{(c,\infty)}(\mathcal{M}),
\end{displaymath}
a contradiction.

\section{Basis properties}

In this section we derive basis properties for the first components of the eigenvectors of $\mathcal{M}$ for which the corresponding eigenvalues lie in the interval $(c,\infty)$.

\begin{cor}\label{c4}
Let $\mathcal{M}$ be a top-dominant or diagonally-dominant operator matrix and let
$C\in\mathcal{B}(\mathcal{H}_2)$. If $A$ has compact resolvent, and $\{(x_n,K_cx_n)^T\}_{n=1}^\infty$ are orthonormal eigenvectors corresponding to the eigenvalues of $\mathcal{M}$ which are greater than $c$, then the set $\{x_n\}_{n=1}^\infty$ forms a Riesz
basis for $\Dom(K_c)$.
\end{cor}
\begin{proof}
The $\{(x_n,K_cx_n)^T\}_{n=1}^\infty$ form an orthonormal basis for $\mathcal{L}_{(c,\infty)}(\mathcal{M})$, and for any $x\in\Dom(K_c)$ we therefore have
\begin{displaymath}
\left(
\begin{array}{c}
x\\
K_cx
\end{array} \right) = \sum \beta_n\left(
\begin{array}{c}
x_n\\
K_cx_n
\end{array} \right)\quad\textrm{where}\quad\beta_n
= \bigg\langle\left(
\begin{array}{c}
x\\
K_cx
\end{array} \right),\left(
\begin{array}{c}
x_n\\
K_cx_n
\end{array} \right)\bigg\rangle.
\end{displaymath}
Therefore $x = \sum\beta_nx_n$, and the sequence $\{\beta_n\}_{n=1}^\infty$ is unique
because the vectors $(x_n,K_cx_n)^T$ are a basis for $\mathcal{L}_{(c,\infty)}(\mathcal{M})$. We deduce
that the $\{x_n\}_{n=1}^\infty$ form a basis for $\Dom(K_c)$. It remains to show that the $\{x_n\}_{n=1}^\infty$ form a Riesz sequence. We have
\begin{equation}\label{riesz1}
\Vert x\Vert^2 + \Vert K_cx\Vert^2 =
\Vert\sum\beta_nx_n\Vert^2 + \Vert K_c\sum\beta_nx_n\Vert^2  =
\sum\vert\beta_n\vert^2,
\end{equation}
and therefore
\begin{equation}\label{riesz2}
\sum\vert\beta_n\vert^2 \le (1+\Vert K_c\Vert^2)\Vert\sum\beta_nx_n\Vert^2.
\end{equation}
From \eqref{riesz1} and \eqref{riesz2} we have
\begin{displaymath}
(1+\Vert K_c\Vert^2)^{-1}\sum\vert\beta_n\vert^2 \le \Vert\sum\beta_nx_n\Vert^2 \le \sum\vert\beta_n\vert^2,
\end{displaymath}
thus the $\{x_n\}_{n=1}^\infty$ form a Riesz sequence, and therefore a Riesz basis for $\Dom(K_c)$.
\end{proof}

\begin{prop}\label{thm11}
Let $\mathcal{M}$ be a top-dominant or diagonally-dominant operator matrix. Let $A$ have compact resolvent with eigenvalues $\{\mu_n\}_{n=1}^\infty$, \begin{equation}\label{ass0}
\dist[\mu_n,\sigma(A)\backslash\{\mu_n\}]\to\infty\quad\textrm{as}\quad n\to\infty,
\end{equation}
and let $C\in\mathcal{B}(\mathcal{H}_2)$. Let $\{\lambda_n\}_{n=1}^\infty$ be the eigenvalues of $\mathcal{M}$ in the interval $(c,\infty)$. If $E$ and $F_n$ are the spectral measures corresponding to $A$ and $S(\lambda_n)$ respectively, and $\gamma_n = \dist[\lambda_n,\sigma(\mathcal{M})\backslash\{\lambda_n\}]/2$, then we have
\begin{equation}\label{measures}
\Vert E(\{\mu_{\kappa+n}\}) - F_n(\Delta_n)\Vert \to 0\quad\textrm{as}\quad n\to\infty,\quad\textrm{where}\quad\Delta_n = (-\gamma_n,\gamma_n).
\end{equation}
\end{prop}
\begin{proof}
We denote by $\Gamma_n$ the circle with centre $\lambda_n$ and radius $\gamma_n$. It follows from \eqref{maththias} that
\begin{equation}\label{ass}
\mu_{\kappa+n} \le \lambda_n \le \mu_{\kappa+n} + a + \frac{ac+b-a^2}{\mu_{\kappa+n}-c} + \mathcal{O}\Big(\frac{1}{\mu_{\kappa_n}^2}\Big);
\end{equation}
see \cite[Corollary 4.4]{math}. From \eqref{ass0} and \eqref{ass} we deduce that $\dist[\lambda_n,\sigma(\mathcal{M})\backslash\{\lambda_n\}]\to\infty$ as $n\to\infty$, therefore, for all sufficiently large $n\in\mathbb{N}$, $\mu_{\kappa + n}$ is the only element from $\sigma(A)$ which lies inside the circle $\Gamma_n$. For $z\in\Gamma_n$, we set
\begin{displaymath}
\delta_n(z) = \frac{a}{\lambda_n-c} + \frac{\vert az + b\vert}{\dist[z,\sigma(A)](\lambda_n-c)},
\end{displaymath}
also from \eqref{ass0} and \eqref{ass} we deduce that
\begin{equation}\label{deltan}
\delta_n:=\max_{z\in\Gamma_n}\delta_n(z)\to 0\quad\textrm{as}\quad n\to\infty.
\end{equation}
Similarly to the proof of Theorem \ref{thm0}, we have for all $x\in \Dom (\vert A\vert^{\frac{1}{2}})$
\begin{displaymath}
\vert\langle(C - \lambda_n I)^{-1}B^*x,B^*x\rangle\vert
\le\delta_n(z)\langle\vert A - zI\vert^\frac{1}{2}x,\vert A - zI\vert^\frac{1}{2}x\rangle.
\end{displaymath}
For all sufficiently large $n\in\mathbb{N}$ we have $\delta_n< 1$, and by \cite[Lemma VI.3.1]{katopert} there exists a family of operators $G_n(z)\in\mathcal{B}(\mathcal{H}_1)$, parameterised by $z\in\Gamma_n$, with the property $\Vert G_n(z)\Vert\le \delta_n<1$, and such that
\begin{displaymath}
-\langle(C - \lambda_n I)^{-1}B^*x,B^*y\rangle = \langle G_n(z)\vert A - \lambda I\vert^{\frac{1}{2}}x,\vert A - \lambda I\vert^{\frac{1}{2}}y\rangle.
\end{displaymath}
Similarly to the proof of Theorem \ref{thm0}, we obtain the following
\begin{displaymath}
S(\lambda_n) -(z-\lambda_n)I = \vert A - zI\vert^\frac{1}{2}\big((A - zI)\vert A - zI\vert^{-1} + G_n(z)\big)\vert A - zI\vert^\frac{1}{2},
\end{displaymath}
and we deduce that $(z - \lambda_n)\in\rho(S(\lambda_n))$ for all $z\in\Gamma_n$. For any $x,y\in\mathcal{H}_1$ the inner product $\langle E(\{\mu_{\kappa+n}\})x - F_n(\Delta_n)x,y\rangle$ is equal to
\begin{displaymath}
-\frac{1}{2i\pi}\int_{\Gamma_n}\langle (A - \zeta I)^{-1}x - (S(\lambda_n) - (\zeta-\lambda_n)I)^{-1}x,y\rangle~d\zeta.
\end{displaymath}
The inner product inside this integral equals
\begin{displaymath}
\Big\langle \Big[I - \big(I + \vert A - \zeta I\vert(A - \zeta I)^{-1}G_n(\zeta)\big)^{-1}\Big]\vert A-\zeta I\vert(A - \zeta I)^{-1}\vert A - \zeta I\vert^{-\frac{1}{2}}x,\vert A - \zeta I\vert^{-\frac{1}{2}}y\Big\rangle,
\end{displaymath}
and using the Neumann series we see that this inner product is equal to
\begin{displaymath}
\sum_{m=1}^\infty\Big\langle \big(\vert A - \zeta I\vert(\zeta I - A)^{-1}G_n(\zeta)\big)^{m}\vert A-\zeta I\vert(A - \zeta I)^{-1}\vert A - \zeta I\vert^{-\frac{1}{2}}x,\vert A - \zeta I\vert^{-\frac{1}{2}}y\Big\rangle,
\end{displaymath}
and therefore
\begin{align*}
\vert\langle E(\{\mu_{\kappa+n}\})x - &F_n(\Delta_n)x,y\rangle\vert\\ &\le\frac{\Vert x\Vert\Vert y\Vert}{2\pi\dist[\Gamma_n,\sigma(A)]}\sum_{m=1}^\infty\int_{\Gamma_n}\big\Vert\vert A - \zeta I\vert(A - \zeta I)^{-1}G_n(\zeta)\big\Vert^m~d\zeta\\
&\le \frac{\vert\Gamma_n\vert\Vert x\Vert\Vert y\Vert}{2\pi\dist[\Gamma_n,\sigma(A)]}\sum_{m=1}^\infty\max_{z\in\Gamma_n}\Vert G_n(z)\Vert^m\\
&\le \frac{\gamma_n\Vert x\Vert\Vert y\Vert}{\dist[\Gamma_n,\sigma(A)]}\sum_{m=1}^\infty\delta_n^m.
\end{align*}
It follows from \eqref{ass} that $\gamma_n/\dist[\Gamma_n,\sigma(A)]\to 1$, and therefore with some $M\in\mathbb{R}$ independent of $n\in\mathbb{N}$, we obtain
\begin{displaymath}
\vert\langle E(\{\mu_{\kappa+n}\})x - F_n(\Delta_n)x,y\rangle\vert \le M\frac{\delta_n}{1 - \delta_n}\Vert x\Vert\Vert y\Vert,
\end{displaymath}
and the result follows from \eqref{deltan}.
\end{proof}

\begin{defn}
Let $\{x_n\}_{n=1}^\infty$ and $\{y_n\}_{n=1}^\infty$ be sequences in a Hilbert space with $\{y_n\}_{n=1}^\infty$ orthonormal. Then $\{x_n\}_{n=1}^\infty$ is a Bari sequence with respect to $\{y_n\}_{n=1}^\infty$ if
\begin{displaymath}
\sum_{n=1}^\infty\Vert y_n - x_n\Vert^2 < \infty.
\end{displaymath}
If $\{x_n\}_{n=1}^\infty$ is a Bari sequence with respect to $\{y_n\}_{n=1}^\infty$ where $\{x_n\}_{n=1}^\infty$ and $\{y_n\}_{n=1}^\infty$ form a basis, then $\{x_n\}_{n=1}^\infty$ is a Bari basis with respect to $\{y_n\}_{n=1}^\infty$.
\end{defn}

\begin{cor}\label{thm2}
Let $\mathcal{M}$ be a top-dominant or diagonally-dominant operator matrix. Let $A$ have compact resolvent with eigenvalues $\{\mu_n\}_{n=1}^\infty$, $\dist[\mu_n,\sigma(A)\backslash\{\mu_n\}]\to\infty$, and let $C\in\mathcal{B}(\mathcal{H}_2)$.
Suppose for some $N\in\mathbb{N}$ and with $\kappa$ as in Corollary \eqref{compactresolvent}, the eigenvalues $\{\mu_{\kappa+n}\}_{n=N}^\infty$ are simple, and
\begin{equation}\label{sum}
\sum_{n=1}^\infty\frac{1}{(\mu_{n+1} - \mu_{n})^2} < \infty.
\end{equation}
If $\{y_n\}_{n=1}^\infty$ are orthonormal eigenvectors of $A$, and $\{(x_n,K_cx_n)^T\}_{n=1}^\infty$ are orthonormal eigenvectors corresponding to the eigenvalues of $\mathcal{M}$ which are greater than $c$, then
$\{x_n\}_{n=1}^\infty$ is a Bari sequence with respect to $\{y_{\kappa+n}\}_{n=1}^\infty$.
\end{cor}
\begin{proof}
From Proposition \ref{thm11}, we have $\Vert E(\{\mu_{\kappa+n}\}) - F_n(\Delta_n)\Vert<1$ for all sufficiently large
$n\in\mathbb{N}$. We have $x_n\in\Dom(S(\lambda_n))$ and $S(\lambda_n)x_n = 0$; see \cite[proof of Proposition 2.2]{math}
and \cite[proof of Proposition 4.4]{math0}. Therefore $x_n\in F_n(\Delta_n)\mathcal{H}_1$. Thus we have
$E(\{\mu_{\kappa+n}\}) x_n\ne 0$ for all sufficiently large $n\in\mathbb{N}$. For $n\ge N$ we have
$E(\{\mu_{\kappa+n}\})x_n = \langle x_n,y_{\kappa+n}\rangle y_{\kappa+n}$, and thus $E(\{\mu_{\kappa+n}\}) x_n = y_{\kappa+n}/\Vert E(\{\mu_{\kappa+n}\}) x_n\Vert$. With the notation of Proposition \ref{thm11}, we obtain
\begin{align*}
\Vert y_{\kappa+n} - x_n\Vert &= \Vert E(\{\mu_{\kappa+n}\}) (y_{\kappa+n} - x_n)\Vert + \Vert(I - E(\{\mu_{\kappa+n}\}) )(y_{\kappa+n} - x_n)\Vert\\
&= \Vert E(\{\mu_{\kappa+n}\}) x_n/\Vert E(\{\mu_{\kappa+n}\}) x_n\Vert - E(\{\mu_{\kappa+n}\}) x_n\Vert\\
&+ \Vert(I - E(\{\mu_{\kappa+n}\}))x_n\Vert\\
&\le 2\Vert(I - E(\{\mu_{\kappa+n}\}))x_n\Vert\\
&= 2\Vert(F_n(\Delta_n)  - E(\{\mu_{\kappa+n}\}) )x_n\Vert\\
&\le\frac{2M\delta_n}{1 - \delta_n}.
\end{align*}
For some $\tilde{N}\ge N$ we have
\begin{displaymath}
\sum_{n=1}^\infty\Vert y_{\kappa+n} - x_n\Vert^2
\le\sum_{n=1}^{\tilde{N}-1}\Vert y_{\kappa+n} - x_n\Vert^2 + 4M^2\sum_{n=\tilde{N}}^\infty\bigg(\frac{\delta_n}{1 - \delta_n}\bigg)^2.
\end{displaymath}
That the second term on the right-hand side is finite follows from \eqref{ass} and \eqref{sum}.
\end{proof}

\begin{rem}
If $\kappa = 0$ in Corollary \ref{thm2} then the $\{x_n\}_{n=1}^\infty$ is a Bari basis with respect to $\{y_{n}\}_{n=1}^\infty$.
\end{rem}

\section{An example from magnetohydrodynamics}

The following operator appears in magnetohydrodynamics in the space $L^2_{\rho}(0,1)\times L^2_{\rho}(0,1)\times L^2_{\rho}(0,1)$, where $L^2_{\rho}(0,1)$ denotes the $L^2$-space with weight $\rho$, and $D$ is the differential operator $-id/dx$. We consider the block operator matrix $\mathcal{M}_0$ given by
\begin{displaymath}
\left(\begin{array}{lll}
        \rho^{-1}D\rho(\upsilon_a^2 + \upsilon_s^2)D + k^2\upsilon_a^2 & (\rho^{-1}D\rho(\upsilon_a^2 + \upsilon_s^2) +
	ig)k_\perp & (\rho^{-1}D\rho\upsilon_s^2 + ig)k_\parallel\\
	k_\perp((\upsilon_a^2 + \upsilon_s^2)D - ig) & k^2\upsilon_a^2 + k_\perp^2\upsilon_s^2 & k_\perp k_\parallel\upsilon_s^2\\
	k_\parallel(\upsilon_s^2D - ig) & k_\perp k_\parallel\upsilon_s^2 & k_\parallel^2\upsilon_s^2
      \end{array}\right).
\end{displaymath}
The operator $\mathcal{M}_0$ on the domain $(W_{\rho}^{2,2}(0,1)\cap W_{0,\rho}^{1,2}(0,1))\times W_{\rho}^{1,2}(0,1)\times W_{\rho}^{1,2}(0,1)$ is essentially self-adjoint; see \cite[Example 4.5]{math}. The eigenvalue problem $\mathcal{M}u = \lambda u$ describes the oscillations of a hot compressible gravitating plasma layer in an ambient magnetic field; see \cite{at} and references therein. The first component of the vector $u$ satisfies Dirichlet boundary conditions, $\rho(x)$ the equilibrium density of the plasma, $\upsilon_a(x)$ the Alfv\'{e}n speed, $\upsilon_s(x)$ the sound speed, $k_\perp(x)$ and $k_\parallel(x)$ are the coordinates of the wave vector with respect to the field allied orthonormal bases, $k(x)^2 = k_\perp(x)^2 + k_\parallel(x)^2$, and $g$ is the gravitational constant.
The essential spectrum of this operator is precisely the range of the functions $\upsilon_a^2k_\parallel$ and
$\upsilon_a^2\upsilon_s^2k_\perp/(\upsilon_a^2+\upsilon_s^2)$; see \cite[Section 5]{at}.

The operator $\mathcal{M}_0$ forms a top-dominant operator matrix, with
\begin{displaymath}
A := \rho^{-1}D\rho(\upsilon_a^2 + \upsilon_s^2)D + k^2\upsilon_a^2:(W_{\rho}^{2,2}(0,1)\cap W_{0,\rho}^{1,2}(0,1))\to L^2_{\rho}(0,1),
\end{displaymath}
and
\begin{displaymath}
C:=\left(\begin{array}{ll}
	k^2\upsilon_a^2 + k_\perp^2\upsilon_s^2 & k_\perp k_\parallel\upsilon_s^2\\
	k_\perp k_\parallel\upsilon_s^2 & k_\parallel^2\upsilon_s^2
      \end{array}\right):L^2_{\rho}(0,1)\times L^2_{\rho}(0,1)\to L^2_{\rho}(0,1)\times L^2_{\rho}(0,1).
\end{displaymath}
We see that $A$ is a Sturm-Liouville operator with Dirichlet boundary conditions. The operator $C$ is self-adjoint and bounded, with
\begin{displaymath}
c = \max\Bigg\{\frac{k^2(\upsilon_a^2+\upsilon_s^2)}{2}+\sqrt{\frac{k^4(\upsilon_a^2+\upsilon_s^2)^2}{4}-
k^2k_\parallel^2\upsilon_a^2\upsilon_s^2}\Bigg\},
\end{displaymath}
and with
\begin{align*}
a&=\max\Bigg\{\frac{(\upsilon_a^2+\upsilon_s^2)^2k_\perp^2+\upsilon_s^4k_\parallel^2}{\upsilon_a^2+\upsilon_s^2}\Bigg\},\\
b&= \max\Bigg\{\max \Big\{k^2g^2-\frac{g}{\rho}(\rho((\upsilon_a^2+\upsilon_s^2)k_\perp+\upsilon_s^2k_\parallel))'\Big\} - a\min \Big\{k^2\upsilon_a^2\Big\},0\Bigg\},
\end{align*}
$A$ and $B^*$ satisfy \eqref{domcons}; see \cite[Example 4.15]{math}. If $\rho$ and $(\upsilon_a^2 + \upsilon_s^2)$ are positive and belong to $C^2([0,1])$, then using the Liouville transform (see for example \cite[Section 2.5.1]{pry}), we see that the eigenvalues of $A$ are precisely those of a Sturm-Liouville operator in normal form with Dirichlet boundary conditions and bounded potential. Thus the eigenvalues of $A$ satisfy
\begin{displaymath}
\dist[\mu_n,\sigma(A)\backslash\{\mu_n\}]\to\infty\quad\textrm{and}\quad\sum_{n=1}^\infty\frac{1}{(\mu_{n+1} - \mu_{n})^2} < \infty.
\end{displaymath}

\begin{thm}
Let $\rho$, $(\upsilon_a^2 + \upsilon_s^2)\in C^2([0,1])$ and $\rho>0$, $\upsilon_a^2 + \upsilon_s^2>0$ on $[0,1]$, and let $\mathcal{M}$ be the closure of $\mathcal{M}_0$. There exists closed bounded operator $K_c$ which satisfies \eqref{angularoperator}, and $\codim(\Dom(K_c)) = \dim(\mathcal{L}_{(-\infty,0)}(S(\tilde{c})))<\infty$. If $\{(x_n,K_cx_n)^T\}_{n=1}^\infty$ are orthonormal eigenvectors corresponding to the eigenvalues $\{\lambda_n\}_{n=1}^\infty$ of $\mathcal{M}$ which are greater than $c$, then the set $\{x_n\}_{n=1}^\infty$ forms a Riesz basis for $\Dom(K_c)$. Moreover, if $\{y_n\}_{n=1}^\infty$ are orthonormal eigenvectors of $A$, then $\{x_n\}_{n=1}^\infty$ is a Bari sequence with respect to $\{y_{\kappa+n}\}_{n=1}^\infty$.
\end{thm}

\end{document}